\documentclass[a4paper,11pt]{article}
\usepackage{amssymb}
\def \Z {\mathbb{Z}}

\def \A {{\mathbb A}}

\def \L {{\cal L}}
\def \GCSL {{\sf GCSL}}
\def \GCSG {{\sf GCSG}}
\def \NCA {{\sf NCA}}
\def\qed{\hfill $\Box$}
\newtheorem{theorem}{Theorem}
\newtheorem{proposition}[theorem]{Proposition}
\newtheorem{lemma}[theorem]{Lemma}
\newtheorem{corollary}[theorem]{Corollary}
\newenvironment{proof}{\normalsize {\sc Proof}:}{{\hfill $\Box$}}
\newenvironment{proofof}[1]{\normalsize {\sc Proof of #1}:}{{\hfill $\Box$}}
\usepackage{latexsym}

\title{Groups that do and do not have growing context-sensitive word
problem}

\author{Derek F. Holt, Sarah Rees and Michael Shapiro}

\date{\today}

\begin{document}

\maketitle

\begin{abstract}
We prove that a group has word problem that is a growing
context-sensitive language precisely if its word problem can be solved
using a non-deterministic Cannon's algorithm (the deterministic
algorithms being defined by Goodman and Shapiro in
\cite{GoodmanShapiro}). We generalise results of \cite{GoodmanShapiro}
to find many examples of groups not admitting non-deterministic
Cannon's algorithms.  This adds to the examples of
Kambites and Otto in~\cite{KambitesOtto} of groups
separating context-sensitive and growing context-sensitive word problems,
and provides a new language-theoretic separation result.
\end{abstract}

\section{Introduction}
The purpose of this note is to extend the results in
Sections 6 and 7 of~\cite{GoodmanShapiro}.  That article
described a linear time algorithm, which we call {\it Cannon's
algorithm}, which generalised Dehn's algorithm for solving the word
problem of a word-hyperbolic group. Many examples of groups that possess
such an algorithm were provided, alongside proofs that
various other groups do not.

The Cannon's algorithms described in ~\cite{GoodmanShapiro} are deterministic,
but there was some brief discussion at the end of Section 1.3
of~\cite{GoodmanShapiro}
of non-deterministic generalisations, and the close connections between
groups with non-deterministic Cannon's algorithms and those with
{\em growing context-sensitive} word problem.
(Non-deterministic Cannon's algorithms and growing context-sensitive
languages are defined in Section~\ref{sec:gcs} below.)
It was clear that an appropriate modification of the theorems in Section 7
of~\cite{GoodmanShapiro} should imply that various
groups, including direct products $F_m \times F_n$ of free groups
with $m>1$ and $n \ge 1$ and other examples mentioned below,
had word problems that were context-sensitive but not growing
context-sensitive.  We provide that modification in this paper.

Our examples are not the first to separate context-sensitive and
growing context-sensitive word problem. For
$F_m \times F_n$ with $m,n>1$ has recently been proved context-sensitive but
not growing context-sensitive by Kambites and Otto
in~\cite{KambitesOtto}, using somewhat different methods;
they failed to resolve this question for $F_m \times F_1
= F_m \times \Z$ with $m>1$, which is now covered by our result. 
And our result has further application, since,
according to~\cite[Section~7]{KambitesOtto}, the fact that
$F_2 \times F_1$ does not have growing context sensitive word problem
implies that the class of growing context-sensitive languages is properly
contained in the language class $\L(\mathsf{OW\!\!-\!\!auxPDA}(poly,log))$.
We refer the reader to~\cite{KambitesOtto} for a definition of this class,
and for citations.

In Section~\ref{sec:gcs} of this article, we provide definitions of the
classes of context-sensitive and growing context-sensitive languages, and of
non-deterministic Cannon's algorithms. 
We then show in Theorem \ref{thm:gcs_can} that the set of formal languages
defined by non-deterministic Cannon's
algorithms is the same as the set of growing context-sensitive languages
that contain the empty word. Hence the word problem for a group can be
solved using a non-deterministic Cannon's algorithm precisely if
it is a growing context-sensitive language.

Two different versions of Cannon's algorithms, known as
{\em incremental} and {\em non-incremental}, are defined
in~\cite{GoodmanShapiro}, and both of these are deterministic.
In \cite[Proposition 2.3]{GoodmanShapiro} it is shown that the language of an
incremental Cannon's algorithm is also the language of a
non-incremental Cannon's algorithm.  In \cite[Theorem 3.7]{KambitesOtto}, it is
shown that the class of languages defined by
non-incremental Cannon's algorithms is exactly the class of
Church-Rosser languages.  It is pointed out in~\cite{KambitesOtto}
(following Theorems 2.5 and 2.6)
that the class of Church-Rosser languages is contained in the class of
growing context-sensitive languages.  It follows that every group with
an incremental or non-incremental Cannon's algorithm also has a
non-deterministic Cannon's algorithm. (This is not obvious, because
replacing a deterministic algorithm with a non-deterministic algorithm
with the same reduction rules could conceivably result in extra words
reducing to the empty word.) In particular, all of the examples
shown in~\cite{GoodmanShapiro} to have (incremental) Cannon's algorithms
have growing context-sensitive word problem.

It follows from the (known) fact that the class of growing context-sensitive
languages is closed under inverse homomorphism that the property of a finitely
generated group having growing context-sensitive word problem does not depend
on the choice of finite semigroup generating set of $G$. Some other closure
properties of the class of growing context-sensitive groups are mentioned
at the end of Section~\ref{sec:gcs}.

In Section~\ref{sec:nondet}, we generalise Theorems~7.1 and~7.2
of~\cite{GoodmanShapiro} to prove that there does not exist a
non-deterministic Cannon's algorithm for a group satisfying the hypotheses of
those theorems; that is, we prove the following.
\begin{theorem}
\label{thm:nondet1}
Let $G$ be a group that is generated as a semigroup by the finite set
$\mathcal{G}$, and suppose that, for each $n \ge 0$, there are sets
$S_1(n), S_2(n)$ in $G$ satisfying
\begin{enumerate}
\item[(1)] each element of $S_i(n)$ can be represented by a word over
$\mathcal{G}$ of length at most $n$,
\item[(2)]
there are constants $\alpha_0>0,\alpha_1>1$ such that
for infinitely many $n$, $|S_i(n)| \geq \alpha_0\alpha_1^n$, and
\item[(3)]
each element of $S_1(n)$ commutes with each element of $S_2(n)$.
\end{enumerate}
Then $G$ cannot have a non-deterministic Cannon's algorithm
over any finite semigroup generating set.
\end{theorem}
\begin{theorem}
\label{thm:nondet2}
Let $G$ be a group that is generated as a semigroup by the finite set
$\mathcal{G}$, and suppose that, for each $n \ge 0$, there are sets
$S_1(n), S_2(n)$ in $G$ satisfying
\begin{enumerate}
\item[(1)] each element of $S_i(n)$ can be represented by a word over
$\mathcal{G}$ of length at most $n$,
\item[(2)]
there are constants $\alpha_0>0,\alpha_1>1,\alpha_2 >0$ such that
for all sufficiently large $n$, $|S_1(n)| \geq \alpha_0\alpha_1^n$ and
$|S_2(n)| \geq \alpha_2n$, and
\item[(3)]
each element of $S_1(n)$ commutes with each element of $S_2(n)$.
\end{enumerate}
Then $G$ cannot have a non-deterministic Cannon's algorithm
over any finite semigroup generating set.
\end{theorem}

As a corollary we see that a group $G$ satisfying the conditions of
Theorem~\ref{thm:nondet1} or~\ref{thm:nondet2}
cannot have word problem that is a growing context-sensitive language.

Theorems~7.1 and~7.2 of~\cite{GoodmanShapiro} are followed by a number
of further theorems and corollaries (numbered 7.3 -- 7.12), which
provide a wide variety of examples of groups that satisfy the criteria
of Theorems~7.1 or~7.2 and hence have no deterministic Cannon's
algorithm.  These examples include $F_2 \times \Z$, braid groups on
three or more strands, Thompson's group $F$, Baumslag-Solitar groups
$\langle a,t, \mid ta^p t^{-1} = a^q \rangle$ with $p \ne \pm{q}$, and
the fundamental groups of various types of closed 3-manifolds.  We can
conclude immediately from our Theorems~\ref{thm:nondet1}
or~\ref{thm:nondet2} that none of these examples have
non-deterministic Cannon's algorithms or growing context sensitive
word problems.

Note that a number of these examples, such as $F_2 \times \Z$ and
braid groups~\cite{ECHLPT}, are known to be automatic and hence to have
context-sensitive word problem \cite{Shapiro}.

On the other hand the article~\cite{GoodmanShapiro} gives a wealth of
examples of groups that do have Cannon's algorithms, and hence have
growing context-sensitive word problems. These include
word-hyperbolic, nilpotent and many relatively hyperbolic groups.

We would like to acknowledge the contribution of Oliver Goodman, many of whose
ideas are visible in the arguments of this paper.

\section{Growing context-sensitive languages and non-deterministic
Cannon's algorithms}
\label{sec:gcs}

We start with the necessary definitions.
A {\em phase-structured grammar} is a quadruple
$(N,X,\sigma, \mathcal{P})$, where $N$ and $X$ are finite sets known
respectively as the {\it non-terminals} and {\em terminals},
$\sigma \in N$ is the {\em start symbol}, and $\mathcal{P}$
is the set of {\em productions}. The productions have the form
$u \rightarrow v$ with $u \in \A^+$ and $v \in \A^*$, where $\A := N \cup X$.

The grammar is {\em context-sensitive} if $|u| \le |v|$ for all
productions $u \rightarrow v$.
It is {\em growing context-sensitive} if, in addition,
for all productions $u \rightarrow v$,
$\sigma$ does not occur in $v$ and either $u = \sigma$ or $|u| <  |v|$.

As is customary, to allow for the possibility of having the empty word
$\epsilon$ in a (growing) context-sensitive language (defined below),
we also allow there to be a production $\sigma \rightarrow \epsilon$,
provided that $\sigma$ does not occur in the right hand side of any
production.

For $u,v \in A^*$, we write $u \rightarrow^* v$ if we can derive $v$ from
$u$ by applying a finite sequence of productions to the substrings of $u$.
The {\em language} of the grammar is the set of words  $ w \in X^*$ with
$\sigma \rightarrow^* w$. A {\em growing context-sensitive language} (\GCSL)
is a language defined by a growing context-sensitive grammar (\GCSG).

There is some information on this class of languages in~\cite{Buntrock}.
Other possibly relevant references are~\cite{BZ94} and~\cite{BO95}.
It is proved in~\cite{BZ92} that the \GCSL s
form an {\it abstract family of languages} which implies, in particular,
that they are closed under inverse homomorphisms and intersection with
regular languages. This in turn implies that the property of the word
problem of a finitely generated group $G$ being a \GCSL\ is
independent of the choice of finite generating set for $G$, and that
this property is inherited by finitely generated subgroups of $G$.

A {\em non-deterministic Cannon's Algorithm} (\NCA) is defined to be a
triple $(X,\A, \mathcal{R})$,
where $X \subseteq \A$ are finite alphabets, and $\mathcal{R}$ is a set of
rules of the form $v \rightarrow u$ with $u,v \in A^*$ and $|v| > |u|$.
For a non-deterministic Cannon's Algorithm, we drop the restriction
imposed in~\cite{GoodmanShapiro} for the deterministic case
that no two rules are allowed to have the same left hand sides.
We allow some of the rules  $v \rightarrow u$ to be {\em anchored} on the left,
on the right, or on both sides, which means that they can only be applied 
to the indicated subword $v$ in words of the form $vw$, $wv$, and $v$,
respectively, for words $w \in \A^*$.  The language of the \NCA\ is
defined to be the set of words $w \in X^*$ with $w \rightarrow^* \epsilon$.

The similarity between \GCSG s and \NCA s is obvious --
replacing the productions
$u \rightarrow v$ of the former by rules $v \rightarrow u$ of the latter
almost provides a correspondence between them. Apart from the reversed
direction of the derivations, there are two principal differences. The first
is that in a \GCSG\ the derivations start with $\sigma$,
whereas with a \NCA\ the significant chains of substitutions end with the
empty word $\epsilon$.
The second is that \NCA s may have anchored rules, whereas
the definition of a \GCSG\ does not allow for the possibility of anchored
productions.

These differences turn out not to be critical, and
in this section we shall prove the following result.

\begin{theorem}~\label{thm:gcs_can}
Let $L$ be a language over a finite alphabet $X$ with $\epsilon \in L$.
Then $L$ is growing context-sensitive if and
only it is defined by a non-deterministic Cannon's Algorithm.
\end{theorem}

To handle the anchoring problem,
let us define an {\it extended} \GCSG\ to be one in which some of
the productions $u \rightarrow v$ may be left anchored, right
anchored, or left and right anchored, which means that
they can only be applied to the indicated subword $u$ in words of the form
$uw$, $wu$ or $u$, respectively, for words $w \in \A^*$. Note that we do
not allow productions with $u=\sigma$ to be anchored, and neither is there any
need to do so, because they can only be used as the initial productions in
derivations of words in the language.

The following proposition tells us that allowing anchored productions does not
augment the class of \GCSL s.

\begin{proposition}\label{no_anch}
If a language $L \subseteq X^*$ is the language defined by an
extended \GCSG, then $L$ is also defined by a standard \GCSG.
\end{proposition}
\begin{proof}
Suppose $L$ is defined by the extended \GCSG\ $(N,X,\sigma, \mathcal{P})$
and let $\A := N \cup X$.

We first replace the grammar by one in which $u \in N^+$ for all productions
$u \rightarrow v$.
In other words, no terminal occurs in the left hand side of any
production.

To do this, we introduce a new set $\tilde{X}$ of non-terminals in one-one
correspondence with $X$. For a word $w \in A^*$, let $\tilde{w}$ be the
result of replacing every terminal $x$ in $w$ by its corresponding
non-terminal $\tilde{x}$. We replace each production $u \rightarrow v$
by a collection of productions of the form $\tilde{u} \rightarrow v'$,
where $v'$ ranges over all possible words obtained by replacing some of
the terminals $x$ that occur in $v$ by their corresponding non-terminals
$\tilde{x}$. This achieves the desired effect without altering the language
of the grammar.

After making that change, we introduce three new sets of non-terminals,
$\verb+^+ N$, $N \verb+^+$ and $\verb+^+ N \verb+^+$, each in
one-one correspondence with $N$.
For a word $w \in \A^+$, we define $\verb+^+ w$ as follows.
Let $w = xv$ with $x \in \A$, $v \in \A^*$.
If $x \in N \setminus \{\sigma \}$, then we set $\verb+^+ w = (\verb+^+ x)v$,
where $\verb+^+ x$ is the symbol in $\verb+^+ N$ that corresponds to $x$.
Otherwise, if $x \in X \cup \{\sigma \}$, we set $\verb+^+ w = w$. 
We define $w \verb+^+$ and $\verb+^+ w \verb+^+$ similarly.
(Note that the new symbols in $\verb+^+ N \verb+^+$ are only needed here when
$|w| = 1$.)

Now each production of the grammar of the form $\sigma \rightarrow v$ is
replaced by $\sigma \rightarrow \verb+^+ v \verb+^+ $.
For each non-anchored production $u \rightarrow v$ with $u \ne \sigma$,
we keep this production and also introduce new (non-anchored) productions 
$\verb+^+ u \rightarrow \verb+^+ v$,
$u \verb+^+ \rightarrow v \verb+^+$, and
$\verb+^+ u \verb+^+  \rightarrow \verb+^+ v \verb+^+ $.
Each left anchored production $u \rightarrow v$ is replaced by
the two productions $\verb+^+ u \rightarrow \verb+^+ v$ and
$\verb+^+ u \verb+^+  \rightarrow \verb+^+ v  \verb+^+$,
and similarly for right anchored productions.  A left and right anchored
production $u \rightarrow v$ is replaced by the single production
$\verb+^+ u \verb+^+  \rightarrow \verb+^+ v  \verb+^+$.

The effect of these changes is that the symbols in $\verb+^+ N$ can only occur
as the leftmost symbol of a word in a derivation starting from $\sigma$ and,
similarly, those in $N \verb+^+$ can only occur as the rightmost symbol.
Conversely, in any word $w \ne \sigma$ that occurs in such a derivation, if
the leftmost symbol of $w$ is a non-terminal then it lies in $\verb+^+ N$,
and similarly for the rightmost symbol.
The symbols in
$\verb+^+ N \verb+^+$ can only arise as the result of an initial derivation
of the form $\sigma \rightarrow \verb+^+ v \verb+^+ $ with $|v| = 1$.   
So the productions that were initially anchored can now only be applied
in a production at the left or right hand side of the word, and so we
have effectively replaced anchored productions by non-anchored ones that
behave in the same way.
Hence the language defined by this grammar is the same as that defined
by the original grammar.
\end{proof}

A natural question that arises at this point is whether anchored rules
in a \NCA\ can be dispensed with in a similar
fashion. The following simple example shows that this is not possible.
Let $X= \{ x \}$. Then $L := \{ x \}$ is the language of the \NCA\
with $\A = X$ and the single rule $x \rightarrow \epsilon$ that is left and
right anchored. A \NCA\ without anchored rules recognising $L$
would have to contain the rule $x \rightarrow \epsilon$, but then $L$ would
also contain $x^n$ for all $n > 0$.
However, the proof of Theorem~\ref{thm:gcs_can} that follows shows that we
can make do with non-anchored rules together with left and right anchored
rules of the form $v \rightarrow \epsilon$.

\begin{proofof}{Theorem~\ref{thm:gcs_can}}
Suppose that $L \subseteq X^*$ is the language defined by the \GCSG\ 
$(N,X,\sigma,\mathcal{P})$ and that $\epsilon \in L$.
We define a \NCA\ $(X,\A,\mathcal{R})$ with
$\A = X \cup N \setminus \{\sigma\}$, where the rules $\mathcal{R}$
are derived from the productions $\mathcal{P}$ as follows.
Productions of the form $\sigma \rightarrow v$ with $v \ne \epsilon$
are replaced by left and
right anchored rules $v \rightarrow  \epsilon$. All productions
$u \rightarrow v$ with $u \ne \sigma$ are replaced by the rule
$v \rightarrow u$.  It is easily seen that derivations of words in
$L$ using $\mathcal{P}$ correspond exactly, but in reverse order, to
reductions of words to $\epsilon$ using $\mathcal{R}$, so the language of
$(X,\A,\mathcal{R})$ is equal to $L$.

Conversely, suppose that $L$ is the language of the \NCA\
$(X,\A,\mathcal{R})$. Then we define an extended \GCSG\ 
$(N,X,\sigma,\mathcal{P})$ as follows. We introduce $\sigma$ as
a new symbol and put $N := (\A \setminus X) \cup \{ \sigma \}$.
We make $\sigma \rightarrow \epsilon$ a production of $\mathcal{P}$,
and the remaining productions are derived from the rules $\mathcal{R}$
as follows.
Rules of the form $v \rightarrow u$ with $u \ne \epsilon$ are replaced
by productions $u \rightarrow v$, where anchored rules are replaced by
correspondingly anchored productions.

For a non-anchored rule $v \rightarrow \epsilon$, we introduce
non-anchored productions $x \rightarrow xv$ and $x \rightarrow vx$
for all $x \in \A$, together with a production $\sigma \rightarrow v$.
For a left-anchored rule $v \rightarrow \epsilon$, we introduce
left-anchored productions $x \rightarrow vx$ for all $x \in \A$
together with a production $\sigma \rightarrow v$. Right-anchored
rules of this form are handled similarly.
Finally, for a left and right anchored rule $v \rightarrow \epsilon$,
we introduce only the production $\sigma \rightarrow v$.

Again there is an order-reversing correspondence between reductions
to $\epsilon$ using $\mathcal{R}$ and derivations using $\mathcal{P}$,
so the language of this grammar is equal to $L$ and, by
Proposition~\ref{no_anch}, we may replace it by a standard \GCSG\ 
with language $L$. 
\end{proofof}

We saw earlier that the property of a group $G$ having growing context word
problem is independent of the chosen semigroup generating set of $G$,
and is closed under passing to finitely generated subgroups.
It is proved in~\cite[Theorems 2.11,2.13]{GoodmanShapiro} that
the property of $G$ having a deterministic Cannon's Algorithm is
preserved under taking free products and under passing
to overgroups of finite index. These proofs work equally well for
non-deterministic Cannon's algorithms, and so having growing context
sensitive word problem is also closed under these operations.
But we shall see in the next section that $F_2 \times F_2$
does not have growing context sensitive word problem. Hence we see
that the class of groups
with growing context sensitive word problem is not
closed under direct products.

\section{Groups without non-deterministic Cannon's algorithm}
\label{sec:nondet}
This section is devoted to the proofs of Theorems~\ref{thm:nondet1}
and~\ref{thm:nondet2}, which extend Theorems~7.1 and 7.2 of~\cite{GoodmanShapiro}.

Our proofs of Theorems~\ref{thm:nondet1} and~\ref{thm:nondet2} are
modifications of the original proofs in~\cite{GoodmanShapiro}.  We
assume the reader is familiar with that work and has it available for
reference.  We will show how to modify those arguments so that they
can be extended to the non-deterministic case.

The argument of~\cite{GoodmanShapiro} starts by examining the
reduction of a word $w_0$ to $w_n$ by repeated applications of the
length-reducing rules, which are referred to as the {\em history to time
$n$ of $w_0$}.  These words are then displayed laid out in
successive rows in a rectangle.  The place in each row where a rule is
to be applied is marked with a {\em substitution line}.  The letters
resulting from the substitution occupy the space on the next row below
this line and are each given equal width.  The authors introduce the notion
of a {\em splitting path} which is a decomposition of such a rectangle
into a right and left piece, together with combinatorial information
on that decomposition.  Given two histories, $v_0,\dots,v_r$ and
$w_0,\dots,w_s$, if these have equivalent splitting paths, then the
left half of the first rectangle can be spliced together with the
right half of the second rectangle in a way which produces the history
of the reduction starting with $v_0^-w_0^+$ and ending with
$v_r^-w_s^+$. (The super-scripts denote the left and right halves of
these words.)  The combinatorics of splitting paths are such that in
certain key situations, exponentially many cases are forced to share
only polynomially many equivalence classes of splitting paths.  The
hypotheses of Theorems~7.1 and 7.2 of~\cite{GoodmanShapiro} assure a
supply of exponentially many commutators, each of which must reduce to
the empty word.  One shows that two of these can be spliced together
to produce a word which does not represent the identity, but which
also reduces to the empty word.  This is a contradiction.

Essentially we are able to work with the same definitions of 
histories, splitting
paths and their details, and equivalence of splitting paths as
\cite{GoodmanShapiro}, but need to introduce
a definition of equivalence of histories, and re-word and re-prove
some of the technical results involving these concepts.
With those revisions,
we shall see that minor variations of the original proofs verify the
non-deterministic versions of the theorem.



We now describe how to modify the proofs of Theorems~7.1 and 7.2
of~\cite{GoodmanShapiro} to prove Theorems~\ref{thm:nondet1}
and~\ref{thm:nondet2}.

Given a non-deterministic Cannon's algorithm and a word $w_0$, there
is no longer a unique history to time $n$ of $w_0$. We can call
any sequence of words $w_0,w_1,\ldots w_n$ produced as the algorithm
makes $n$ substitutions on $w_0$ {\em a} history, although it is no
longer valid to call it {\em the} history.

The definitions of a ``diagram'', a ``substitution line'', and the 
``width'' of a letter need no modification, nor do Lemmas~6.1
and~6.2 of~\cite{GoodmanShapiro} which relate the width of a letter to
its generation.

The definition of a ``splitting path'' needs no
modification. Lemma~6.4 of~\cite{GoodmanShapiro} states that a letter
of generation $g$ has a splitting path ending next to it of length at
most $2g+2$.  This remains true {\em if we choose the history
appropriately.}  We now show that we can do this.

Observe that in the non-deterministic case if a word $w$ contains as
disjoint substrings two left-hand sides, say $u$ and $u'$ of the rules
$u\to v$ and $u'\to v'$, then these two substitutions can be carried
out in either order, i.e., either as 
$$xuyu'z \to xvyu'z \to xvyv'z$$ 
or as 
$$xuyu'z \to xuyv'z \to xvyv'z.$$  
Now consider two histories, 
$$h_1 = w_0,\dots,xuyu'z,\, xvyu'z,\, xvyv'z, \dots, w_n$$ 
and 
$$h_2 = w_0,\dots, xuyu'z ,\, xuyv'z ,\, xvyv'z,\dots, w_n.$$
(Corresponding ellipses stand for identical sequences.)  We will say
that these are {\em equivalent reductions\begin{footnote}{The notion
of equivalent reductions is not to be confused with the notion of
equivalent histories defined below.  Accordingly, we will briefly
refer to histories as {\em reductions} to distinguish these
equivalence relations.}\end{footnote}} and this generates an
equivalence relation on reductions starting with $w_0$ and ending with
$w_n$.  We can then speak of {\em corresponding substitutions} in
equivalent reductions.  Notice that corresponding substitution lines
in equivalent reductions have the same width, occur at the same
position horizontally, consume the same letters with the same widths
and generations and produce the same letters with the same width and
generation.

Given a history $w_0,\dots,w_n$, there is a partial ordering of its
substitutions which is generated by the relation $s_1 \prec s_2$ if
$s_2$ consumes a letter produced by $s_1$.  The relationship $\prec$
is visible in the diagram of the history in that $s_1 \prec s_2$ if
and only if there is a sequence of substitution lines with
horizontally overlapping segents starting at $s_1$ and ending at
$s_2$.  In particular, $\prec$-incomparable substitution lines do not
overlap in horizontal position, except possibly at their endpoints.
(We will omit further mention of this possible exception.)  Because of
this, given two $\prec$-incomparable substitutions $s_1$ and $s_2$ we
either have $s_1$ {\em lying to the left of} $s_2$ or {\it vice
versa}.  Notice that if $s_1$ lies to the left of $s_2$, then this is
so for the corresponding substitutions in any equivalent reduction.

\begin{lemma}
Suppose that $h_1 = w_0,\dots,w_n$ is a reduction containing the
substitutions $s_1$ and $s_2$ in which $s_1$ takes place before $s_2$
and $s_1$ and $s_2$ are $\prec$-incomparable.  Then there is an
equivalent reduction $h_2$ in which the substitution corresponding to
$s_1$ takes place after that corresponding to $s_2$.
\end{lemma}

\begin{proof}
Note that $s_1$ and $s_2$ do not overlap horizontally.  Thus, if
these two substitutions take place at successive words of $h_1$, we
are done. 

Suppose now that no $\prec$-ancestor of $s_2$ takes place later than
$s_1$.  In that case, $s_1$ can be interchanged with the immediately
preceding substitution, thus reducing by 1 the number of substitutions
occurring between $s_1$ and $s_2$.  Continuing this in this way
produces the previous case.

Finally, suppose there is $s_3 \prec s_2$ with $s_3$ occurring later
than $s_1$.  Let us suppose that $s_3$ is the earliest such.  Then
$s_1$ and $s_3$ are $\prec$-incomparable, for otherwise we would have
$s_1\prec s_2$.  Applying the previous case allows us to move $s_3$
prior to $s_1$, thus reducing by 1 the number of $\prec$-ancestors of
$s_2$ lying between $s_1$ and $s_2$.  Continuing in this way produces
the previous case.
\end{proof}

\begin{corollary} Suppose that $h_1 = w_0,\dots, w_n$ is a reduction and
that $\Sigma_1$ and $\Sigma_2$ are disjoint sets of substitutions in
$h_1$ with the property that no substitution of $\Sigma_1$ is
$\prec$-comparable with any substitution of $\Sigma_2$.  Then there is an
equivalent reduction $h_2$ in which every substitution of $\Sigma_1$
takes place before every substitution of $\Sigma_2$. \qed
\end{corollary}

In view of this, by passing to an equivalent reduction, we may assume
that if $s_1$ and $s_2$ are $\prec$-incomparable and $s_1$ takes place
to the left of $s_2$, then $s_1$ takes place prior to $s_2$.

Using this assumption on our choice of history justifies the statement
in the proof of Lemma~6.4 that, ``When this happens it can only be
with substitutions to the left in the upper half and to the right in
the lower.'' The proof now goes through as before.

The definitions of ``details'' and equivalence of splitting paths need
no modification.  Remark~6.5 of~\cite{GoodmanShapiro} gives a bound on
the number of equivalence classes  of length $n$. This remains
valid.  (However, we could simplify the details by ceasing to record
the $W-1$ letters to the left/right in a right/left segment as part of
the details. These are only needed to ensure that substitutions take
place in the intended order in the deterministic case in Lemma 6.6.)

Lemma~6.6 of~\cite{GoodmanShapiro} states that if two histories,
$v_0,\dots,v_r$ and $w_0,\dots,w_s$ have equivalent splitting paths
then these can be spliced to form the history starting with
$v_0^-w_0^+$ and ending with $v_r^-w_s^+$.  In our case, we need to
modify the statement of Lemma~6.6 of~\cite{GoodmanShapiro} to say
``Then {\em a} history of $v_0^- w_0^+$ up to a suitable time ...''
rather then ``{\em the} history'', because this history may not be
unique. With that change, Lemma~6.6 remains true.

The paragraph after the proof of Lemma~6.6 no longer applies at all; that is,
$v_r$ is not necessarily determined even in a weak sense by $v_0$.

Section~6.1 adapts these methods to respect subword boundaries.  Lemma
6.7 of~\cite{GoodmanShapiro} says that if $w_0$ is a subword of $v_0$
of length $N$ and $w_t$ has length at least $2W-1$, then the reduction
of $w_0$ to $w_t$ (and hence the reduction of $v_0$ to $v_t$) has a
splitting path in one of at most $C_1N^{C_2}$ classes.  The results of
this section up to and including this Lemma remain valid.

Lemma~6.8 of~\cite{GoodmanShapiro} discusses the way that the choice
of $u_0$ affects the subword reduction of $v_0$ in the word $u_0v_0$.
This does not make sense as stated for a non-deterministic algorithm,
because reduction of $u_0v_0$ to $u_tv_t$ no longer gives the word $v_t$
as a function of $u_0$.

In order to state an analogous result, we need an appropriate notion of
equivalence of histories.
For fixed $v_0$ and variable $u_0$, we define two histories
$u_0 v_0 \rightarrow u_t v_t$  and $u_0' v_0 \rightarrow u_t' v_t'$ to be
{\em equivalent} if either:

(i) $l(v_t) < W$ and $v_t = v_t'$; or\\
(ii) $l(v_t), l(v_t') \ge W$, the first $W-1$ letters of $v_t$ and $v_t'$
are the same,
and the two histories have equivalent splitting paths that begin at the same
place in $v_0$ and end immediately after the first $W-1$ letters of
$v_t$ and $v_t'$.

Then it follows from Lemma~6.6 of~\cite{GoodmanShapiro} that,
if $u_0 v_0 \rightarrow u_t v_t$  and
$u_0' v_0 \rightarrow u_t' v_t'$ are equivalent histories,
then there is also a history
$u_0' v_0 \rightarrow u_t' v_t$, to which both are equivalent.
(Note: this notation may seem to imply that all of these histories have the
same length  $t$, but of course they need not. All three lengths might be
different!)

Then the proof of Lemma~6.8 of~\cite{GoodmanShapiro} shows that,
for fixed $v_0$ of length $N$,
there are at most $C_0 N^C$ equivalence classes of histories
$u_0 v_0 \rightarrow u_t v_t$ for variable $u_0$.

We turn now to the proof of our Theorem~\ref{thm:nondet1}.  As we have
already seen, the equivalent properties of having growing context
sensitive word problem and having non-deterministic Cannon's
algorithms hold independently of the finite semigroup generating set
$\mathcal{G}$, so we only need to prove the non-existence of the
Cannon's algorithm over $\mathcal{G}$. We can further assume that $1
\in \mathcal{G}$ so that any element which is represented by a word of
length less than or equal to $n$ is also represented by a word of
length $n$.  (It is also easily seen directly that the hypotheses of
Theorems~\ref{thm:nondet1} and~\ref{thm:nondet2} do not depend on the
choice of generators.)  Assume for a contradiction that the hypotheses
of that theorem hold and that there exists a non-deterministic
Cannon's algorithm over a working alphabet $\mathbb{A}$ that contains
$\mathcal{G}$.


First we choose a specific value of $n>3W$ that is large enough for this
stronger version of
(1) to apply,
$|S_i(n)| \geq \alpha_0 \alpha_1^n$,
and also such that $n$ is big enough so that that this exponential
function is larger than a particular polynomial function that comes out of
some of our technical lemmas. More precisely, we require
$$\frac{1}{2} \alpha_0 \alpha_1^n > C_1 n^{C_2+2}|\mathbb{A}|^{6W}C_0n^C,$$
where $C,C_0,C_1$ and $C_2$ are the constants defined in Lemmas 6.7 and 6.8.

For $i=1,2$, let $T_i$ be a set of words of length at least $3W$ and at
most $n$ representing the elements of $S_i(n)$.
Since each element of $S_1(n)$ commutes with each element of $S_2(n)$,
we have $u_0v_0u_0^{-1}v_0^{-1} =_G 1$ for all $u_0 \in T_1$ and
$v_0 \in T_2$, and hence this word can be reduced, not necessarily uniquely,
to the empty word by means of the Cannon's algorithm. 
For ease of notation we write $x_0$ for $u_0^{-1}$ and $y_0$ for $v_0^{-1}$.
For each such $u_0$ and $v_0$, we choose some sequence of substitutions
that reduces $u_0v_0x_0y_0$ to the empty word
and let $u_tv_tx_ty_t$ be the word that we get after applying
$t$ substitutions in this sequence.

For such a commutator, we run the algorithm to the point where for the
first time $v_t$ and $x_t$ both have length less than $3W$.
At that time the longer one has length in the
range $[2W,3W-1]$.
Note that $t$ depends on $u_0$ and $v_0$, and where it is used
below it should seen that way, and not as a constant.

First we assume that for at least half of the pairs
$(u_0,v_0) \in T_1 \times T_2$,  we have $l(v_t) \geq l(x_t)$.
We shall deal with the opposite case later.

{\em Step 1.}
Now we fix a $v_0 \in T_2$, chosen such that $l(v_t) \geq l(x_t)$ for at least
half of the words $u_0 \in T_1$, and we let $U$ be the set of all words
$u_0$ with this property. Then $|U| \geq \frac{1}{2} \alpha_0 \alpha_1^n$.

{\em Step 2.}
Since we have $l(v_t) \geq 2W$, we can apply Lemma 6.7, which says
that we can choose a splitting path for the subword history
$v_0,v_1,\ldots v_t$ in one of
polynomially many equivalence classes ($C_1n^{C_2}$, to be precise).
That is to say there are polynomially many sets of details that can describe
such a splitting path. This remains true (with the number increased to
$C_1n^{C_2+2}$) if we add to the detail the information that tells us where
within $v_0$ the splitting path begins and where within $v_t$ the
splitting path ends. We call that the {\em extended detail}. 

{\em Step 3.}
Now since $v_0$ is fixed we have a well defined map $u_0 \mapsto v_tx_ty_t$.
(Recall that, although there may be many possible reduction
sequences for $u_0v_0x_0y_0$, we arbitrarily chose some fixed sequence
for each $u_0$.)
Now we apply our amended version of Lemma 6.8 with $u_0v_0x_0$ in place of
$u_0$ and $y_0$ in place of $v_0$. This tells us there are at
most polynomially ($C_0n^C$ in fact) many equivalence classes of histories
$u_0v_0x_0y_0 \rightarrow u_tv_tx_ty_t$.
But $u_0$ comes from a set of exponential size (at least
$\frac{1}{2} \alpha_0 \alpha_1^n$),
which we have chosen to be bigger than the appropriate
polynomial, through our choice of $n$.
So we have a large set (of size more than $C_1n^{C_2+2}$)
of $u_0 \in U$ that give rise to the same words $v_t$ and $x_t$,
and with the property that the histories
$u_0v_0x_0y_0 \rightarrow u_t(u_0) v_tx_ty_t(u_0)$
(where $v_t$ and $x_t$ are fixed, but $u_t$ and $y_t$ depend on $u_0$)
are all equivalent, in the sense defined above in our comments about
the amended Lemma~6.8.

But, as we also noted above, Lemma~6.6 implies that these histories are
also all equivalent to histories
$u_0 v_0 x_0 y_0 \rightarrow u_t(u_0) v_t x_t y_t$
for the same fixed $y_t$.
Note that in these two equivalent histories
$u_0 v_0 x_0 y_0 \rightarrow u_t(u_0) v_t x_t y_t(u_0)$ and
$u_0 v_0 x_0 y_0 \rightarrow u_t(u_0) v_t x_t y_t$,
the parts of the two histories to the left of the splitting line are the same
except for the number of steps in which the words remain constant.
So we can use essentially the same splitting paths as we chose in Step~2 for
the second history.

{\em Step 4.}
Since the number of $u_0$ giving rise to equivalent histories in Step~3 is
greater than the number of equivalence classes of positioned splitting paths
for $v_t$ in Step~2,
we can choose $u_0, u_0' \in U$ such that
$u_0 v_0 x_0 y_0 \rightarrow u_t v_t x_t y_t$  and
$u_0' v_0 x_0 y_0 \rightarrow u_t' v_t' x_t' y_t'$ are equivalent
histories with $v_t' = v_t$ and $x_t' = x_t$, and such that the subhistories
$v_0 \rightarrow v_t$ and $v_0 \rightarrow v_t'$
in the two histories contain equivalent
splitting paths, which start at the same position in $v_0$ and end in the
same position in $v_t = v_t'$.

By the remark above, the second of these histories (and hence also the first!)
is equivalent to a history
$u_0' v_0 x_0 y_0 \rightarrow u_t' v_t' x_t' y_t$,
which still contains an equivalent splitting path through
$v_0 \rightarrow v_t'$.
Now we can do our splicing, and apply Lemma~6.6 of~\cite{GoodmanShapiro}
to produce a history
$u_0 v_0^- v_0'^+ x_0' y_0' \rightarrow u_t v_t^- v_t'^+ x_t' y_t =
u_t v_t x_t y_t$.

But $u_t v_t x_t y_t$ is part of the originally chosen history that reduces
the commutator
$u_0 v_0 x_0 y_0$ to the empty word, so there exists a history that reduces
$u_0 v_0^- v_0'^+ x_0' y_0'$ to the empty word, a contradiction, because
this is not the identity element of the group.

In the second case (not considered in detail in
Section~7 of~\cite{GoodmanShapiro}) where $l(x_t) \geq l(v_t)$
for at least half the pairs $(u_0,v_0) \in T_1 \times T_2$,
rather than fix $v_0 \in T_2$ we fix $u_0 \in T_1$
in a  similar way, and let $V$ be the possible $v_0$ from which, together with
the chosen $u_0$, we get $l(x_t) \geq l(v_t)$.
Then we apply Lemma~6.7 to the subword histories
$x_1,\ldots x_t$.
Now we look at the map $v_0 \mapsto u_tv_tx_t$.  The analogue of Lemma~6.8
applied to
$u_tv_t$ tells  us that $u_t$ can take polynomially many values, and we
see that we get a large set of possible $v_0 \in V$ corresponding to a single 
$u_tv_tx_t$. Hence we can choose $v'_0$ mapping such that
$u_tv_tx_t=u'_tv'_tx'_t$ and such that the subword histories $x_1,\ldots x_t$
and $x'_1,\ldots x'_t$ have the same extended details, and we can splice.
Hence we see that the algorithm should rewrite
$u_0v_0x_0^-x'^+_0y'_0$ to $u_tv_tx_t^-x'^+_ty'_t$.
Since $x_0=x'_0$,
the first of these two words is equal to $u_0v_0x_0y'_0$, equal in the group to $v_0v'^{-1}_0$, so non-trivial. But the second word is equal to
$u'_tv'_tx'_ty'_t$, which rewrites to the trivial word. Hence
again we get our contradiction.

Modifying the proof of Theorem~7.2 from~\cite{GoodmanShapiro} in
the same way, we arrive at a proof of Theorem~\ref{thm:nondet2}.

Basically we choose $n_1, n_2$ with
\[ \frac{1}{2} \alpha_0\alpha_1^{n_1} >
C_1n_2^{C_2+2}|\mathbb{A}|^{6W}C_0n_2^C,\quad
\frac{1}{2}\alpha_2 n_2 > C_1n_1^{C_2+2}|\mathbb{A}|^{6W}C_0n_1^C \]
which we can do, for example, by first setting $n_2$ equal to some polynomial
function in $n_1$ so that (2) is satisfied for all $n_1$, and then
choosing $n_1$ big enough so that (1) holds.
(But notice that we need Hypothesis (2) in the statement of the theorem to be
satisfied for these particular values of $n_1$ and $n_2$, which is
why we have assumed this hypothesis for all
integers $n >0 $ rather than for infinitely many such $n$, which was sufficient
for Theorem~\ref{thm:nondet1}.)

As in the proof of Theorem~\ref{thm:nondet1},
for $i=1,2$, we choose $T_i$ to be a set of words of length at
least $3W$ and at most $n$ that represent the elements of $S_i$. 
Those conditions will now ensure  each $\frac{1}{2}|T_i|$ ($i=1,2$) is bounded
below by the appropriate polynomial function of  $n_2,n_1$ which allows us
to find $u_0,u'_0$.
So in the case where $v_t$ is longer than $u_t$ we can do just what we did in the first
case of Theorem~\ref{thm:nondet1}, since $T_1$ is big enough we can find a big enough set of
elements of $U$ mapping to the same  $v_tx_ty_t$.

And in the second case we have $T_2$ big enough, and so we can follow the
argument used in the second case of the proof of Theorem~\ref{thm:nondet1}


\begin{thebibliography}{99}
\bibitem{Buntrock} Gerhard Buntrock,
Home Page Growing Context Sensitive Languages,\\
{\small
\verb+http://www.itheoi.mu-luebeck.de/pages/buntrock/research/gcsl.html+}.

\bibitem{BZ92} Gerhard Buntrock and Krzysztof Lory\'s,
On Growing Context-Sensitive Languages,
in 
Automata, Languages and Programming, 19th International Colloquium,
{\em Automata, Languages and Programming, 19th International Colloquium,
ICALP92, Vienna, Austria, July 13-17, 1992, Proceedings},
ed. Werner Kuich,
Lecture Notes in Computer Science, vol 623, Springer 1992, 77--88.

\bibitem{BZ94} Gerhard Buntrock and Krzysztof Lory\'s,
The variable membership problem: succinctness versus complexity,
in 
{\em STACS 94, 11th Annual Symposium on Theoretical Aspects of Computer
               Science, Caen, France, February 24-26, 1994, Proceedings},
ed. Patrice Enjalbert, Ernst W. Mayr and Klaus W. Wagner,
Lecture Notes in Computer Science, vol 775, Springer 1994, 595--606.

\bibitem{BO95} Gerhard Buntrock and Friedrich Otto,
Growing context-sensitive languages and Church-Rosser Languages,
in
{\em STACS 95, 12th Annual Symposium on Theoretical Aspects of Computer Science,
 Munich, Germany, March 2-4, 1995. Proceedings},
ed. Ernst W. Mayr and Claude Puech,
Lecture Notes in Computer Science, vol 900, Springer 1995, 313--324.

\bibitem{ECHLPT} D.B.A.\ Epstien, J.W.\ Cannon, D.F.\ Holt, S.V.F.\
Levy, M.S.\ Paterson, W.P.\ Thurston, Word Processing in Groups, Jones
and Bartlett, 1992.

\bibitem{GoodmanShapiro} O.\ Goodman and M.\ Shapiro,
A generalisation of Dehn's algorithm, arXiv preprint,
\verb+http://front.math.ucdavis.edu/0706.3024+.

\bibitem{KambitesOtto} M.\ Kambites and F.\ Otto,
Church-Rosser groups and growing context-sensitive groups, preprint,
available from
\verb+http://www.maths.manchester.ac.uk/~mkambites/+.

\bibitem{Shapiro} M.\ Shapiro, A note on context-sensitive languages and word 
problems, Internat. J. Algebra Comput. (4) 4 (1994) 493--497.

\end{thebibliography}
\end{document}